\documentclass[10pt,reqno]{amsart}
\usepackage[a4paper,margin=2.8cm,marginparwidth=2cm,bottom=3.5cm]{geometry}
\usepackage{lmodern}
\usepackage[american]{babel}
\usepackage{color}
\usepackage[hidelinks,unicode=true]{hyperref}
\title[On the cutoff phenomenon for fast diffusion and porous medium equations]{On the cutoff phenomenon for\\fast diffusion and porous medium equations}
\date{Winter 2025, revised Autumn 2025, Autumn 2026, compiled \today}%
\author{Djalil Chafaï}%
\address[DC]{DMA, École normale supérieure, 45 rue d'Ulm, 75230, Paris; and CEREMADE, Université Paris-Dauphine, Place du Maréchal de Lattre de Tassigny, 75775, Paris; PSL, CNRS}%
\email{djalil@chafai.net}%
\author{Max Fathi}%
\address[MF]{Université Paris Cité and Sorbonne Université, CNRS, Laboratoire Jacques-Louis Lions and Laboratoire de Probabilités, Statistique et Modélisation, F-75013 Paris, France\\
and DMA, École normale supérieure, Université PSL, CNRS, 75005 Paris, France \\
and Institut Universitaire de France}
\email{mfathi@lpsm.paris}%
\author{Nikita Simonov}%
\address[NS]{Sorbonne Université, Université Paris Cité, CNRS, Laboratoire Jacques-Louis Lions, LJLL, F-75005 Paris, France}%
\email{nikita.simonov@sorbonne-universite.fr}
\numberwithin{equation}{section}
\newtheorem{theorem}{Theorem}[section]%
\newtheorem{lemma}[theorem]{Lemma}%
\newtheorem{remark}[theorem]{Remark}%

\begin{document}

\begin{abstract}
  The cutoff phenomenon, conceptualized at the origin for finite Markov chains,
  states that for a parametric family of evolution equations, started from a point, the distance towards a long time
  equilibrium may become more and more abrupt for certain choices of initial conditions, when the parameter tends to infinity. This threshold phenomenon can be seen as a critical competition between trend to equilibrium and worst initial condition. In
  this note, we investigate this phenomenon beyond stochastic processes, in the context of the analysis of nonlinear partial differential
  equations, by proving cutoff for the fast diffusion and porous medium
  Fokker\,--\,Planck equations on the Euclidean space, when the dimension tends to infinity. We formulate the phenomenon using quadratic Wasserstein distance, as well as using specific relative entropy and Fisher information. Our high dimensional
  asymptotic analysis uses the exact solvability of the model involving
  Barenblatt profiles.
  It includes the Ornstein\,--\,Uhlenbeck dynamics as a special linear case.
\end{abstract}

\subjclass[2020]{
35A23;
35K15;
35K61;
49Q22.
}

\keywords{Cutoff phenomenon; High dimensional phenomenon; Fast diffusion equation; Porous medium equation; Fokker\,--\,Planck equation; Optimal transportation; Functional inequalities}

\maketitle

{\footnotesize\tableofcontents}

\section{Introduction}

\subsection{The cutoff phenomenon}

The cutoff phenomenon is a property of abrupt convergence to the equilibrium that occurs for many families of Markov processes and more generally dynamical systems. It was first highlighted by David Aldous, Persi Diaconis, and Mehrdad Shahshahani in the context of random walks on groups \cite{Aldous83, Diaconis96}, but has appeared in many other contexts since then, including, for instance, interacting particle systems such as exclusion processes \cite{Lacoin16}, Brownian motion on spheres, projective spaces, and on compact Lie groups \cite{Meliot14,MR1306030}, random walks on graphs \cite{LuPe16}, spin systems in statistical physics \cite{LuSl13}, and dynamical systems \cite{caputo-labbe-lacoin}. In a nutshell, a sequence of Markov processes $({(X_t^n)}_{t\geq0})$ with invariant probability measures $(\pi^n)$ satisfies a cutoff at time $t_n$, and with respect to some distance-like quantity ``$\mathrm{dist}$'' on the space of probability measures, when for any $\epsilon > 0$, 
\[
\varliminf_{n\to\infty} 
\sup_{x_0\in S_n}\mathrm{dist}(\mathrm{Law}(X^n_{(1-\epsilon)t_n}), \pi^n) > 0
\quad\text{and}\quad
\varlimsup_{n\to\infty} 
\sup_{x_0\in S_n}\mathrm{dist}(\mathrm{Law}(X^n_{(1+\epsilon)t_n}), \pi^n) = 0
\]
where $S_n$ is a well chosen set of initial conditions. Cutoff is most commonly studied for the total variation distance, but many other distances or divergences are of interest : Hellinger distance, $L^p$ distances, Wasserstein distances, relative entropy, Fisher information, \ldots In most examples, the lower bound at times $(1-\epsilon)t_n$ is not just strictly positive, but actually as large as possible : $+\infty$ for unbounded distances such as Wasserstein distances, $1$ for total variation distance. 

The goal of this work is to show that the cutoff phenomenon also takes place in the context of \emph{nonlinear} evolution equations, beyond the dynamics of Markov processes. More precisely, we establish the cutoff phenomenon for a family of nonlinear parabolic PDE : porous medium and fast-diffusion equations, with drift, in high dimension. Our interest is conceptual, showing that certain ideas coming from probability theory are also relevant in nonlinear analysis, but also specific, since understanding the precise convergence to equilibrium for these PDE is of independent interest. 

\subsection{Fast diffusion, porous medium, Barenblatt profile, Fokker\texorpdfstring{\,--\,}{--}Planck}
This note is about the following nonlinear partial differential evolution equation in $\mathbb{R}^d$, $d\geq1$,
for a parameter $m>0$ :
\begin{equation}\label{eq:FDE}
    \partial_t u
    =\Delta(u^m)
    =\mathrm{div}(mu^{m-1}\nabla u)
   ,\quad t>0,x\in\mathbb{R}^d,
\end{equation}
with a Dirac mass initial condition $u(t=0,\cdot)=\delta_{x_0}$, $x_0\in\mathbb{R}^d$. 
It is a reaction-diffusion equation :
\begin{equation}
\Delta(u^m)=m u^{m-1} \Delta u+m (m-1) u^{m-2} |\nabla u|^2, 
\end{equation}
where $|x|^2=x_1^2+\cdots+x_d^2$ is the squared Euclidean norm. The equation is exactly solvable. Indeed, for $d\geq2$ and $m>\frac{d-2}{d}$, the evolution preserves positivity and mass and admits the unique solution 
\begin{equation}\label{eq:sol}
    u(t,x)=\frac{1}{t^{\alpha d}}B\Bigr(\frac{x-x_0}{t^\alpha}\Bigr),\quad t>0,x\in\mathbb{R}^d,
\end{equation}
where 
\begin{equation}\label{eq:B}
    B(x):={\Bigr(c+\alpha\frac{1-m}{2m}|x|^2\Bigr)_+^{\frac{1}{m-1}}},
    \quad
    \alpha:=\frac{1}{2-d(1-m)}
    =\frac{1}{d(m-\frac{d-2}{d})}>0.
\end{equation}
The constant $c=c_{d,m}>0$ is chosen is such a way that  function $B$, often called the Barenblatt or Barenblatt\,--\,Pattle profile in the analysis of PDE literature, due to \cite{MR0046217,MR0114505}, is a probability density function. Actually the Barenblatt profile matches classical families of probability distributions, see Remarks \ref{rk:student} and \ref{rk:spherical} below. The equation \eqref{eq:FDE} is strongly connected to Sobolev type inequalities. We refer to \cite{MR1940370,MR1986060,bonforte:hal-02887010,MR2286292,zbMATH05058450,zbMATH06622527,zbMATH06431289} for the analysis of PDE point of view, and to \cite{demange-thesis,MR2178944,MR2381156} for more geometric and probabilistic aspects. For stochastic processes associated to \eqref{eq:FDE}, we refer to \cite{MR1027932,MR1469575,MR2421181}.

There are three main regimes depending of the value of the parameter $m$:
\begin{itemize}
\item Case $\frac{d-2}{d}<m<1$. The PDE \eqref{eq:FDE} is known as the fast
  diffusion equation. Since $1-m>0$, the Barenblatt profile $B$ has a negative
  power, has full space support, is heavy tailed, and is nothing else but a
  multivariate Student t-distribution, see Remark \ref{rk:student}.
  It has a finite second moment if and only if
  $m>\frac{d}{d+2}$. Moreover the case $m=\frac{d-1}{d}$ has remarkable
  properties, in particular it gives $\frac{1}{1-m}=d$. Some useful formulas
  shall be given in Lemma \ref{le:Bm<1}.
  \item Case $m=1$. The PDE \eqref{eq:FDE} is the heat or diffusion equation,
    it is linear. The Barenblatt profile $B$ is Gaussian (heat kernel), see
    Section \ref{ss:difflim} for formulas when $m=1$ and $m\to1$.
  \item Case $m>1$. The PDE \eqref{eq:FDE} is known as the porous medium
    equation. Since $m-1>0$, the Barenblatt profile has a positive power, but
    a compact support, related to a finite speed of propagation for the
    evolution equation. The Barenblatt profile is then nothing else but a sort
    of symmetric multivariate or radial Beta distribution supported on a
    sphere of $\mathbb{R}^d$ :
    \begin{equation}
    c^{\frac{1}{m-1}}
    (1-b|x|)^{\frac{1}{m-1}}(1+b|x|)^{\frac{1}{m-1}}    
    \mathbf{1}_{|x|\leq\frac{1}{b}},\quad x\in\mathbb{R}^d,\quad
    b:=\sqrt{\alpha\frac{1-m}{m}c}.
  \end{equation}
  This also appears as the law of the projection of the uniform law on
  spheres, see Remark \ref{rk:spherical}. Some useful formulas are given by
  Lemma \ref{le:Bm>1}. These are not Dirichlet distributions.
\end{itemize}   

Following \cite[eq.~(18)]{MR1940370}, we consider now the Fokker\,--\,Planck version of \eqref{eq:FDE} given by
\begin{equation}\label{eq:FDEFP}
    \partial_tv=\Delta(v^m)+\mathrm{div}(xv),\quad t>0, x\in\mathbb{R}^d,
\end{equation}
with Dirac initial data $v(t=0,\cdot)=\delta_{x_0}$, $x_0\in\mathbb{R}^d$. Its solution is related to the solution $u$ of \eqref{eq:FDE} via
\begin{equation}\label{eq:uv}
    u(t,x)=R(t)^{-d}v\Bigr(\tau(t),\frac{x}{R(t)}\Bigr)
\end{equation}
where $\tau(t):=\log R(t)$ and $R$ is the solution of the ODE $\dot{R}=R^{\frac{\alpha-1}{\alpha}}$, $R(0)=1$, given by
$R(t):=\big(1+\frac{t}{\alpha}\bigr)^\alpha$. In other words, denoting $R_\tau:=R\circ\tau^{-1}$, we have $v(t,x)=(R_\tau(t))^du(\tau^{-1}(t),R_\tau(t)x))$. But since $\tau(t)=\alpha\log(1+\frac{t}{\alpha})$, we get $\tau^{-1}(t)=\alpha(\mathrm{e}^{\frac{t}{\alpha}}-1)$, 
$(R_\tau)(t)=\mathrm{e}^t$, and thus, by \eqref{eq:uv} and  \eqref{eq:B},
\begin{equation}\label{eq:FDEFPv}
v(t,x)=\mathrm{e}^{dt}u(\alpha(\mathrm{e}^{\frac{t}{\alpha}}-1),\mathrm{e}^tx)
  =\frac{1}{(\alpha(1-\mathrm{e}^{-\frac{t}{\alpha}}))^{\alpha d}}
  B\Bigr(\frac{x-\mathrm{e}^{-t}x_0}{(\alpha(1-\mathrm{e}^{-\frac{t}{\alpha}}))^\alpha}\Bigr).
\end{equation}
We get immediately from \eqref{eq:FDEFPv} the long time behavior of the solution $v$: 
\begin{equation}\label{eq:vinf}
     v_\infty(x):=\lim_{t\to\infty}v(t,x)
    =\frac{1}{\alpha^{\alpha d}}B\Big(\frac{x}{\alpha^\alpha}\Bigr)=\left(C+\frac{1-m}{2m}|x|^2\right)_{+}^\frac{1}{m-1}
\end{equation}
($C$ is not $c$). The long time limit no longer depends on the initial condition $x_0$. For a
finer study of this long time behavior, we remark that $v$ given by
\eqref{eq:FDEFPv} and it particular its long time limit remains in the
position-scale family of the Barenblatt profile $B$. This allows to use
formulas for certain distances within this family, in the same spirit as
\cite[Lemma A.5]{MR4528974}, see also \cite[Lemma 3 and 4]{MR2246356}.

\subsection{Distances and divergences, and main result on the cutoff phenomenon}

Recall that if $\mu$ and $\nu$ are probability measures on $\mathbb{R}^d$ with finite second moment, the Monge\,--\,Kantorovich\,--\,Wasserstein quadratic transportation distance $\mathrm{W}_2(\mu,\nu)$ is given by
\begin{equation}
\mathrm{W}_2(\mu,\nu):=
\sqrt{\inf_\pi\int|x-y|^2\mathrm{d}\pi(x,y)}
\end{equation}
where the infimum runs over the set of probability measures on $\mathbb{R}^d\times\mathbb{R}^d$ with marginals $\mu$ and $\nu$. In the same probabilistic functional analytic spirit, and following for instance \cite[Chap.~2]{bonforte:hal-02887010},
the free energy or relative entropy is given by
\begin{equation}\label{eq:tsalis}
    \mathrm{H}_m(f\mid v_\infty):=
    \frac{1}{m-1}\int\left(f(x)^m-v_\infty(x)^m-\frac{1-m}{2}|x|^2(f(x)-v_\infty(x))\right)\mathrm{d}x.
\end{equation}
Similarly, the Fisher information or entropy production along this dynamics is
\begin{equation}\label{eq:fisher}
    \mathrm{I}_m(f\mid v_\infty):=\int f\Bigr|x+\frac{m}{m-1}\nabla f^{m-1}\Bigr|^2\mathrm{d}x.
\end{equation}
The relative entropy and the Fisher information satisfy the following equation along solutions to~\eqref{eq:FDEFP}
\begin{equation}
\frac{\mathrm{d}}{\mathrm{d}t}
\mathrm{H}_m(v(t)\mid v_\infty) ) = - \mathrm{I}_m(v(t)\mid v_\infty)\,.
\end{equation}
Note that when $m<1$, or when $m>1$ and the support of $f$ is included in the one of $v_\infty$, then $\mathrm{H}_m(f\mid v_\infty)$ can be written as the Bregman\,--\,Rényi\,--\,Tsallis divergence associated to~\eqref{eq:FDEFP}:
\begin{equation}\label{eq:tsalisbis}
    \mathrm{H}_m(f\mid v_\infty)=
    \frac{1}{m-1}\int\left(f(x)^m-v_\infty(x)^m-m\,v_\infty^{m-1}(x)(f(x)-v_\infty(x))\right)\mathrm{d}x.
\end{equation}
The function $x\mapsto\frac{x^m-1}{m-1}$ is strictly convex for $m\neq1$. Up to the affine term $\frac{x-1}{m-1}$, whose integral vanishes between probability densities, it gives the classical logarithmic relative entropy as $m\to1$, since $\frac{x^m-x}{m-1}\to x\log x$ (note that $\frac{x^m-1}{m-1}$ itself diverges as $m\to1$ when $x\neq1$). 
When $f$ is a location-scale transformation of the Barenblatt profile $v_\infty$, the expression \eqref{eq:tsalis} is fully computable, while 
\eqref{eq:tsalisbis} is not when $m>1$ due to the cross term
$v_\infty^{m-1}f$.

Our main results below state that in high dimension, the distance of the solution to the equilibrium collapses abruptly at a critical time. We call this threshold phenomenon the \emph{cutoff phenomenon}. It is a nonlinear generalization of the same phenomenon for the Ornstein\,--\,Uhlenbeck process stated in \cite[Theorem 1.1]{MR4528974} see also \cite{MR2260742,MR2203823,MR3770869}. We refer to \cite{MR1306030,MR2375599} for more on the cutoff phenomenon for linear evolution equations associated to Markov processes, and to \cite{chafathi,Salez25,pedrotti2025transportapproachcutoffphenomenon} for recent advances for positively curved linear Markov diffusions.

Let $v(t,\cdot)$ be the solution of the PDE \eqref{eq:FDEFP} in $\mathbb{R}^d$, $d\geq3$, $m>\frac{d-2}{d}$, given by \eqref{eq:FDEFPv}, with point initial condition $x_0\in\mathbb{R}^d$. Let $v(t,\infty)$ be as in \eqref{eq:vinf}, associated to the Barenblatt profile $B$ in \eqref{eq:B}.

We assume that $d\geq3$ because when $\frac{d-2}{d}<m\leq\frac{d}{d+2}$, which is the case for $d=2$, the $\mathrm{W}_2$ distance is infinite. But this constraint on $d$ is not relevant here since we are interested in the regime $d\to\infty$. The main advantage of using this distance is the explicit formula that it gives for position-scale families, that boils down to an explicit (second) moment computation for the Barenblatt profile.

\begin{theorem}[Cutoff phenomenon]\label{th:main}
 Let $\alpha\in(0,1/2)\cup(1/2,+\infty)$ and $r > 0$, and set $m=\frac{(d-2)\alpha +1}{d\alpha}$.\\ 
 Then for all $\mathrm{dist}(\cdot,\cdot)\in\{\mathrm{W}_2^2(\cdot,\cdot),\mathrm{H}_m(\cdot\mid\cdot),\mathrm{I}_m(\cdot\mid\cdot)\}$, and for all $\varepsilon\in(0,1)$,
  \begin{equation}\label{eq:m>1:W2}
    \lim_{d\to\infty}
    \sup_{|x_0| \leq r\sqrt{d}}\mathrm{dist}(v(t_d,\cdot),v_\infty)
    =\begin{cases}
    +\infty & \text{if } t_d=(1-\varepsilon)\frac{\max(1,\alpha)}2\log(d)\\
    0 & \text{if } t_d=(1+\varepsilon)\frac{\max(1,\alpha)}2\log(d)\\
    \end{cases}.  
  \end{equation}
  In particular, the cutoff critical time is of order $\log(d)$.
\end{theorem}

We emphasize that both $v$, $v_\infty$, $x_0$ and $\mathrm{dist}$ depend on the dimension $d$. We discuss the motivation behind choice of the set of initial data $\{|x_0| \leq r\sqrt{d}\}$ in Section \ref{sec:initialdata} below, and we shall discuss what happens for other choices there. Note that the value of $r$ plays no role. 

Note that $m=\frac{(d-2)\alpha+1}{d\alpha}>\frac{d-2}{d}$ for $d\geq2$.

For $\alpha=1/2$, we get the linear Ornstein\,--\,Uhlenbeck dynamics, see Section \ref{ss:difflim}, and we recover here the cutoff obtained in \cite{MR4528974}. 

For $\alpha=1$, we get $m=\frac{d-1}{d}$, an important exponent related to the equality case in Sobolev type functional inequalities, see~\cite{bonforte:hal-02887010}. In particular, as $d\to\infty$, from the Sobolev inequality one get a logarithmic Sobolev inequality, see for instance~\cite{arXiv:2209.08651}, which motivates in part this research.  The motivation for considering the limit $\ d\to\infty $ along curves $ d\mapsto m$ with constant $ \alpha $ arises from the analysis of the spectrum of the linearized operator around the Barenblatt profile. 
When $m < 1$ (that is, $\alpha > \frac{1}{2}$), the spectrum (computed in~\cite{zbMATH02187606}, see also~\cite{zbMATH02155995} for further details and in~\cite{zbMATH06133917} for a different framework) is determined by the values of $m$ and the dimension $d$. It consists of an essential component together with discrete eigenvalues.  In particular, for curves with constant $\alpha$, the spectral gap depends on whether $\alpha < 1$ or $\alpha > 1$. In the porous medium case, $m > 1$ (equivalently $0 < \alpha < \frac{1}{2}$), the spectrum is purely discrete and was computed in~\cite{zbMATH06262313} (see also~\cite{zbMATH04096011} for the one-dimensional case). In this regime, the first spectral gap is independent of $\alpha$, while the second is proportional to $\alpha$. 

In the fast diffusion case $m < 1$, it is worth emphasizing the dependence of the cutoff critical time on the value of $\alpha$. More specifically, on whether $\alpha < 1$ or $\alpha > 1$. This reflects the fact that the spectral gap, as well as the corresponding eigenfunctions, change across these regimes, as first observed in~\cite{zbMATH02187606}.

Instead of fixing the value of $\alpha$, we can instead fix the value of $m$. In that regime, we have

\begin{theorem}[Cutoff phenomenon for fixed exponent]\label{th:fixed_exp}
 Let $m>1$ and $r > 0$, and set $\alpha:=\frac{1}{2-d(1-m)}$.
 Then for all $\mathrm{dist}(\cdot,\cdot)\in\{\mathrm{W}_2^2(\cdot,\cdot),\mathrm{H}_m(\cdot\mid\cdot),\mathrm{I}_m(\cdot\mid\cdot)\}$, and for all $\varepsilon\in(0,1)$,
  \begin{equation}\label{eq:mfixed:W2}
    \lim_{d\to\infty}
    \sup_{|x_0| \leq r\sqrt{d}}\mathrm{dist}(v(t_d,\cdot),v_\infty)
    =\begin{cases}
    +\infty & \text{if } t_d=(1-\varepsilon)\frac12\log(d)\\
    0 & \text{if } t_d=(1+\varepsilon)\frac12\log(d)\\
    \end{cases}.  
  \end{equation}
  In particular, the cutoff critical time is of order $\log(d)$.
\end{theorem}

Note that the prefactor in the mixing time does not depend on $\alpha$ anymore, since fixing $m$ forces $\alpha$ to go to zero. We cannot fix $m < 1$, since as we let $d$ go to infinity eventually the constraint $m > (d-2)/d$ would cease to hold. 

\subsection{Further comments}

More about Barenblatt profiles can be found for instance in \cite{MR1426127,MR0046217,bonforte:hal-02887010}.

\subsubsection{Linearization}

The regime $m=\frac{d-1}{d}\to1^-$ ($\alpha=1$), differs from the regime $m=1$ ($\alpha=\frac{1}{2}$). The latter is up to a scaling the one in \cite[Th.~1.1]{MR4528974} for Gaussians. The fast diffusion equation becomes linear in the high dimensional limit $d\to\infty$ when $m=\frac{d-1}{d}\to1^-$, as well as in the case $m=1$ and fixed $d$, and we recover in both cases an Ornstein\,--\,Uhlenbeck dynamics, see Section \ref{ss:difflim}.

\subsubsection{Linear case or diffusion limit}\label{ss:difflim}
At fixed $d$, when $m=1$ or $m\to1$ we get from \eqref{eq:B}
\begin{equation}
    \alpha=\tfrac{1}{2},\quad B(x)\to(4\pi)^{-\frac{d}{2}}\mathrm{e}^{-\frac{|x|^2}{4}},
\end{equation} in other words we recover in this case the Gaussian heat kernel 
\begin{equation}\label{eq:heatkernel}
    u(t,x)=(4\pi t)^{-\frac{d}{2}}\mathrm{e}^{-\frac{|x-x_0|^2}{4t}},
\end{equation}while \eqref{eq:FDE} becomes the heat equation $\partial_tu=\Delta u$ with Dirac initial data $u(t=0,\cdot)=\delta_{x_0}$.
Similarly, when $m=1$ or $m\to1$, the PDE \eqref{eq:FDEFP} becomes the Fokker\,--\,Planck equation of the standard Ornstein\,--\,Uhlenbeck dynamics,  $\alpha=\frac{1}{2}$, and we recover from \eqref{eq:FDEFPv} and \eqref{eq:heatkernel} the Mehler formula 
\begin{align}
  v(t,x)
  =(2\pi(1-\mathrm{e}^{-2t}))^{-\frac{d}{2}}\mathrm{e}^{-\frac{|x-\mathrm{e}^{-t}x_0|^2}{2(1-\mathrm{e}^{-2t})}},
\end{align}
in other words the density of the Gaussian distribution $\mathcal{N}(\mathrm{e}^{-t}x_0,(1-\mathrm{e}^{-2t})\mathrm{id}_d)$.

\subsubsection{Choice of initial data} \label{sec:initialdata}
We highlighted a particular order of magnitude for the initial data, by taking balls of radius of order $\sqrt{d}$. This choice is because at equilibrium the first moment is also of order $\sqrt{d}$, and hence this order of magnitude is typical at equilibrium. One could easily adapt the results to other polynomial growth rates for the radius (as a function of the dimension). In Theorem \ref{th:fixed_exp} the mixing time would become $\theta \log d$ if the radius scales like $d^{\theta}$. In Theorem \ref{th:main} it would be $\max(\alpha/2, \theta)\log d$.

\subsubsection{Cutoff universality and functional inequalities} 

As already observed in \cite{MR4528974}, it is striking to see how the different distances and divergences match, producing the same critical time for the cutoff, and suggesting the use of functional inequalities : Sobolev type inequality, Talagrand transportation inequality, etc. This is indeed the case for linear diffusions, see \cite{chafathi}, as well as for our nonlinear diffusions, see \cite[Proposition 5.1 and Theorem 5.2]{MR2824568}. 

\subsubsection{Further developments}

For further development, among the numerous open questions, we can mention the non-Dirac initial conditions, the case of the sphere and the image on the Euclidean space by the stereographical projection of the equation $\partial_t u=\Delta(u^m)$, the case of the $p$-Laplacian.
In particular, let us mention that on the sphere of $\mathbb{R}^d$, and for the case $m=\frac{d-1}{d}$, we have the following analogue of the Barenblatt profile, discovered by Jérôme Demange \cite[p.~597]{MR2381156} :
\begin{equation}
    u(t,x)=\Bigr(\frac{\sinh((d-1)t)}{\cosh((d-1)t)-\langle x_0,x\rangle}\Bigr)^d.
\end{equation}

\subsubsection{About methods for cutoffs}

Historically, the first diffusions for which a cutoff phenomenon was established were Brownian motion (heat equation) on compact manifolds with specific symmetries, such as spheres, tori, projective spaces, and classical Lie groups, by using $L^2$ decomposition, see \cite{MR1306030}, and \cite{Meliot14} for further extensions. The symmetries have the advantage to produce exact solvability (explicit formulas).
Regarding non compact spaces, the cutoff for Ornstein\,--\,Uhlenbeck type diffusions (Fokker\,--\,Planck equation) was analyzed using exact formulas or tensorization, see  \cite{MR2260742,MR2203823,MR3770869}. Beyond such exactly solvable situations, and until very recently, the upper and lower bound for cutoff were analyzed separately. Proving sharp lower bounds is often done with more model-specific techniques, and require identifying the worst possible initial data. For the upper bound, it is often enough to get sharp rates of convergence to equilibrium for general initial data, which can be done using coupling techniques, or functional inequalities such as log-Sobolev and Nash inequalities, see for instance \cite{MR1306030} and \cite{chafathi}. Spectral approaches are easiest to implement when proving $L^2$ cutoff \cite{MR2584746}, and $L^p$ cutoff can be deduced from $L^2$ cutoff using the Riesz\,--\,Thorin interpolation theorem \cite{MR2375599}. However, this does not extend to $L^1$ cutoff, and does not apply in the present nonlinear setting. Recently, the case $p=1$ (total variation) was addressed in \cite{Salez25} for non-negatively curved linear diffusions using a new approach based on varentropy, which does not require a separate study of upper and lower bound. The approaches in \cite{MR2375599} and \cite{Salez25} do not provide the value of the critical time. In another direction, the cutoff and the critical time for various distances for positively curved linear diffusions is considered in the recent work \cite{chafathi} under a spectral gap constraint, which goes beyond \cite{MR4528974}. The approach based on transportation and elliptic regularity developed in \cite{chafathi} and later on in \cite{pedrotti2025transportapproachcutoffphenomenon} is particularly appealing in the context of nonlinear PDEs. These recent advances involve classical tools from probabilistic functional analysis such as Bakry\,--\,Émery Gamma calculus, functional inequalities, and transportation of measure. These techniques are well-established in the analysis of fast diffusion and porous medium equations \cite{zbMATH06133917}, so it would be very interesting to explore their usage for the cutoff.

We emphasize that the cutoff phenomenon can be studied with respect to other parameters than the dimension, and more generally for any family of evolution equations with a unique steady state. 

\subsection*{Acknowledgments}
This work has been supported by the Project CONVIVIALITY ANR-23-CE40-0003 of the French National Research Agency. 
We thank the anonymous referee for pointing out several relevant references that were missing in the initial version of this paper.

\subsection*{Artificial intelligence statement}
The mathematical results in arXiv v1 of this article were obtained without the
use of artificial intelligence. Version 2 corrects mathematical errors
identified during the formalization of the results in Lean 4 using Mathlib.
The formalization is available at
\url{https://github.com/djalilchafai/cutoff-fastdiff}. It was developed using
Claude Code with Anthropic's Claude Opus 5.5 model, under the supervision of
the first author. The formal proofs are checked by the Lean kernel. The
authors take full responsibility for the content of the article.

\section{Useful lemmas about Wasserstein distance}

The following lemmas can be of independent interest, beyond their key role for our purposes.

\begin{lemma}[Wasserstein for position-scale transformation]\label{le:wasposca}
  Let $\mu$ and $\nu$ be probability measures on $\mathbb{R}^d$ with finite
  second moment. If $\nu$ is the image or pushforward of $\mu$ by
  an affine map 
  \begin{equation}
    x\mapsto T(x)=Ax+h
  \end{equation}
  where $h$ is a
  vector of $\mathbb{R}^d$ and $A$ is a positive-semidefinite $d\times d$
  symmetric matrix, then
 \begin{equation}\label{eq:W2mom}
    \mathrm{W}_2^2(\mu,\nu)
    =\mathrm{Trace}((A-I)^2M_\mu)
    +2\langle(A-I)m_\mu,h\rangle+|h|^2
 \end{equation}
 where $m_\mu:=\int x\mathrm{d}\mu(x)$ and $M_\mu:=\int
 xx^\top\mathrm{d}\mu(x)$ are the first two moments of $\mu$. Alternatively, 
  \begin{equation}\label{eq:W2cov}
   \mathrm{W}_2^2(\mu,\nu)=\mathrm{Tr}(\Sigma_\mu+\Sigma_\nu-2A\Sigma_\mu)+|m_\mu-m_\nu|^2,
  \end{equation}
 where $\Sigma_\mu$ and $\Sigma_\nu$ are the covariance matrices of $\mu$ and $\nu$, and
  where $m_\nu$ is the mean of $\nu$. 
\end{lemma}

We use this lemma for Barenblatt profiles. 

\begin{proof}
  First of all, an affine map is the gradient of a convex function if and only if the matrix of the linear part is symmetric positive semidefinite.
  Second, the uniqueness in the Brenier theorem on optimal transportation \cite{MR1100809,MR1369395,MR1964483,MR3409718} states that if a transportation map is the gradient of a convex function, then it is the optimal transportation map.
  It remains to use the fact that
  \begin{equation}
   \mathrm{W}_2^2(\mu,\nu)
   =\int|T(x)-x|^2\mathrm{d}\mu(x)
   =\int(|(A-I)x|^2+2\langle(A-I)x,h\rangle+|h|^2)\mathrm{d}\mu(x).
  \end{equation}
  This gives \eqref{eq:W2mom}. To get \eqref{eq:W2cov}, we note that $m_\nu=Am_\mu+h$, $M_\mu=\Sigma_\mu+m_\mu m_\mu^\top$, $\Sigma_\nu=A\Sigma_\mu A$, thus
  \begin{equation}\label{eq:W2A}
   \mathrm{W}_2^2(\mu,\nu)=\mathrm{Tr}(\Sigma_\mu+\Sigma_\nu-2A\Sigma_\mu)+|m_\mu-m_\nu|^2.
  \end{equation}
  We note also by the way that $T(x)=A(x-m_\mu)+m_\nu$.
\end{proof}

It turns out that the location-scale family of a rotationally invariant probability distribution is parametrized by the mean and covariance. We speak about elliptic families or elliptic distributions. Basic examples are given by Gaussian distributions and more generally Barenblatt profiles (see also Remark \ref{rk:student} and Remark \ref{rk:spherical}). For elliptic families, the affine map can be expressed in terms of the covariance matrices, as expressed by the following lemma, that we give for the sake of completeness. The formula for the distance based on means and covariances is well known for Gaussians \cite{MR0752258,MR0745785}. 

\begin{lemma}[Wasserstein for elliptic families]\label{rk:wasell}
  Let $\eta$ be a rotationally invariant probability measure on $\mathbb{R}^d$,
  with zero mean and identity covariance matrix. Let 
  $\mu$ and $\nu$ be two probability measures in the location-scale family of $\eta$, namely
  images or pushforwards of $\eta$ by the affine maps
  \begin{equation}
    x\mapsto T_\mu(x):=m_\mu+\sqrt{\Sigma_\mu}x
    \quad\text{and}\quad
    x\mapsto T_\nu(x):=m_\nu+\sqrt{\Sigma_\nu}x,
  \end{equation}
  where $m_\mu,m_\nu\in\mathbb{R}^d$, and where $\Sigma_\mu$ and $\Sigma_\nu$ are $d\times d$ symmetric
  matrices, $\Sigma_\mu$ positive definite and $\Sigma_\nu$ positive semidefinite. Then $\mu$ and $\nu$ have mean $m_\mu$ and $m_\nu$ and covariance $\Sigma_\mu$ and $\Sigma_\nu$ respectively, and
  \begin{equation}\label{eq:W2sig}
    \mathrm{W}_2^2(\mu,\nu)
    =\mathrm{Tr}\Bigr(\Sigma_\mu+\Sigma_\nu-2\sqrt{\sqrt{\Sigma_\mu}\Sigma_\nu\sqrt{\Sigma_\mu}}\Bigr)+|m_\mu-m_\nu|^2.
  \end{equation}
  In particular, if the covariance matrices commute : $\Sigma_\mu\Sigma_\nu=\Sigma_\nu\Sigma_\mu$, then 
  \begin{equation}
    \mathrm{W}_2^2(\mu,\nu)
    =\mathrm{Tr}\bigr((\sqrt{\Sigma_\mu}-\sqrt{\Sigma_\nu})^2\bigr)+|m_\mu-m_\nu|^2.
 \end{equation}  
\end{lemma}

\begin{proof}
  The rotational invariance of $\eta$ implies that its image by an affine map $x\mapsto m+Cx$ depends on $C$ only via its covariance $CC^\top$. It follows that if we show that the image of $\eta$ by an affine map has same mean and covariance as $\nu$, then it is equal to $\nu$. Let us consider the affine map $x\mapsto T(x):=A(x-m_\mu)+m_\nu$ where $A$ is the positive semidefinite symmetric matrix
  \begin{equation}\label{eq:A}
      A=\sqrt{\Sigma_\mu}^{-1}\sqrt{\sqrt{\Sigma_\mu}\Sigma_\nu\sqrt{\Sigma_\mu}}\sqrt{\Sigma_\mu}^{-1}.
  \end{equation}
  Now the image of $\eta$ by the affine map $T\circ T_\mu$ is nothing else but $\nu$, because the matrix $C:=A\sqrt{\Sigma_\mu}$ satisfies $CC^\top=A\Sigma_\mu A=\Sigma_\nu$. As a consequence, the image of $\mu$ by the affine map $T$ is $\nu$.  
  Finally \eqref{eq:W2sig} follows from the formula \eqref{eq:W2A} of Lemma \ref{le:wasposca} by using the cyclic property of the trace.
\end{proof}

\section{Full-space Barenblatt and proof of Theorem \ref{th:main} (fast diffusion case)}

This section is devoted to the proof of Theorem \ref{th:main} in the fast diffusion case $\alpha>1/2$.

Let us start with a general lemma concerning the second moment and the $\mathrm{L}^m$ norm of the Barenblatt profile. In what follows $\Gamma(z)=\int_0^\infty t^{z-1}\mathrm{e}^{-t}\mathrm{d}t$ denotes the usual Euler Gamma function, defined for any $z\in\mathbb{C}$ such that $\mathrm{Re}(z)>0$.
\begin{lemma}[Full-space Barenblatt profile]\label{le:Bm<1}
 The full-space Barenblatt probability density function in dimension $d\geq1$ and with shape parameter $p>\frac{d}{2}$ and scale parameter $b>0$ is given by
 \begin{equation}\label{eq:Bxc}
   B(x):=\frac{1}{{(c+b|x|^2)}^p},\quad x\in\mathbb{R}^d,\quad\text{where}\quad
   c:=\Bigr(\frac{\pi^{\frac{d}{2}}\Gamma(p-\frac{d}{2})}{b^{\frac{d}{2}}\Gamma(p)}\Bigr)^{\frac{2}{2p-d}}.
 \end{equation}
 Moreover, we also have
 \begin{align}
   \int_{\mathbb{R}^d}|x|^aB(x)\mathrm{d}x
   &=\Bigr(\frac{\pi^{\frac{d}{2}}\Gamma(p-\frac{d}{2})}{b^{p}\Gamma(p)}\Bigr)^{\frac{a}{2p-d}}
   \frac{\Gamma(\frac{d+a}{2})\Gamma(p-\frac{d+a}{2})}{\Gamma(\frac{d}{2})\Gamma(p-\frac{d}{2})},
   \quad\text{for all $-d<a<2p-d$}\label{eq:Bac}\\
   \int_{\mathbb{R}^d}B(x)^m\mathrm{d}x&=\Bigr(\frac{\pi}{b}\Bigr)^{\frac{dp(1-m)}{2p-d}}
    \Bigr(\frac{\Gamma(p)}{\Gamma(p-\frac{d}{2})}\Bigr)^{\frac{2pm-d}{2p-d}}
    \frac{\Gamma\left (pm-\frac{d}{2}\right)}{\Gamma(pm)},\quad\text{for all $p>\frac{d}{2m}$}\label{eq:Bm}.
 \end{align}
 Furthermore, if we fix $\alpha\in\left(\frac{1}{2},\infty\right)$ and set $m=\frac{\alpha(d-2)+1}{\alpha d}$, $p=\frac{1}{1-m}=\frac{d\alpha}{2\alpha - 1}$, $b=\frac{2\alpha-1}{2(\alpha(d-2)+1)}$, then   
 \begin{equation}\label{eq:M2sim}
\int_{\mathbb{R}^d}|x|^2 B(x)\mathrm{d}x =(2\pi\mathrm{e})^{2\alpha-1} d + o_{d\rightarrow\infty} (d)\quad\mbox{and}\quad \int_{\mathbb{R}^d}B(x)^m\mathrm{d}x=\left(2\pi\mathrm{e}\right)^{2\alpha-1} + o_{d\rightarrow\infty}(1)
\end{equation}
Furthermore: 
\begin{equation}\label{difference-fde}
    \int_{\mathbb{R}^d}|x|^2 B(x)\mathrm{d}x - d \int_{\mathbb{R}^d}B(x)^m\mathrm{d}x = O_{d\to\infty}(1)
\end{equation}
\end{lemma}

\begin{proof}
As in \cite{MR3103175}, by using spherical coordinates, the fact that the surface of the unit sphere of $\mathbb{R}^d$, $d\geq1$, is $2\pi^{\frac{d}{2}}/\Gamma(\frac{d}{2})$, and an Euler Beta integral, we get, for all $d\geq1$, $p>\frac{d}{2}$, $-d<a<2p-d$, $b>0$, $c>0$, 
\begin{equation}\label{eq:Ba}
  M_a:=
   \int_{\mathbb{R}^d}|x|^aB(x)\mathrm{d}x
   =\frac{2\pi^{\frac{d}{2}}}{\Gamma(\frac{d}{2})}
    \left(\frac{c}{b}\right)^{\frac{d+a}{2}}
    \frac{\Gamma(\frac{d+a}{2})\Gamma\left (p-\frac{d+a}{2}\right)}{2c^p\Gamma(p)}.
\end{equation}
In particular $B$ is a probability density function if and only if $M_0=1$, namely $c$ is as in \eqref{eq:Bxc}, which gives \eqref{eq:Bac}. 
Finally, we get \eqref{eq:Bm} by using $M_0$ with $pm$ instead of $p$ and the $c$ associated to $p$.
Furthermore, under the current assumptions on $p$, $a$ and $b$, we find that
\begin{equation}\label{eq:Momentsfast}
    M_2= \frac{\pi^{2\alpha-1}}{b^{2\alpha}}\left(\frac{\Gamma(\frac{d}{2(2\alpha-1)})}{\Gamma(\frac{d}{2}\frac{2\alpha}{(2\alpha-1)})}\right)^\frac{2(2\alpha-1)}{d} \frac{d(2\alpha-1)}{d+2-4\alpha}
\end{equation}
where we have used the property $\Gamma(z+1)=z\Gamma(z)$. By noticing that $b^{-2\alpha}=(\alpha2)^{2\alpha}d^{2\alpha}(2\alpha-1)^{-2\alpha}+o_{d\rightarrow\infty}(d^{2\alpha})$ and by using the Stirling formula ($\Gamma(z)\sim \sqrt{2\pi}z^{z-1/2}\mathrm{e}^{-z}$ as $z\rightarrow\infty$ or the following version $\Gamma(sz)^\frac{1}{z}\sim \left(\frac{sz}{e}\right)^s$ which holds for any $s>0$)
we find that
\begin{equation}
     \int_{\mathbb{R}^d}|x|^2 B(x)\mathrm{d}x =(2\pi\mathrm{e})^{2\alpha-1} d + o_{d\rightarrow\infty}(d)\,.
\end{equation}
Similarly, we obtain the $\mathrm{L}^m$-norm of the Barenblatt profile:
\begin{equation}\label{eq:Bmnorm}\begin{split}
N_m:=\int_{\mathbb{R}^d}B(x)^m\mathrm{d}x & =  \left(\frac{\pi}{b}\right)^{(2\alpha-1)}\left(\frac{\Gamma(\frac{d}{2(2\alpha-1)})}{\Gamma(\frac{d}{2}\frac{2\alpha}{(2\alpha-1)})}\right)^\frac{2(2\alpha-1)}{d}\frac{2\alpha(d-2)+2}{d+2-4\alpha} \\
& = \left(2\pi\mathrm{e}\right)^{2\alpha-1} + o_{d\rightarrow\infty}(1)
\end{split}\end{equation}
At the same time, we notice that, by improving a little our estimate on $b=\frac{\beta}{2\alpha d}+O_{d\to\infty}(d^{-2})$, where $\beta=2\alpha-1$ we find that
\begin{equation}\label{last-limit}
\begin{split}
M_2-dN_m&= \frac{M_2}{\beta}\left(\beta -b\frac{2\alpha(d-2)+2}{\beta}\right) \\
&=\frac{M_2}{\beta}\left(\beta-(2\alpha d - 2\beta)\left(\frac{\beta}{2\alpha d}+O_{d\rightarrow\infty}(d^{-2})\right)\right) =\frac{M_2}{\beta} O_{d\to\infty}(d^{-1})\,,
\end{split}
\end{equation}
and we find relation~\eqref{difference-fde} since $M_2=O_{d\to\infty}(d)$ as $d\to\infty$.
\end{proof}

\begin{proof}[Proof of Theorem \ref{th:main} when $\alpha>1/2$ (fast diffusion case)]

The density $v_\infty$ has a finite second moment if and only if $m>\frac{d}{d+2}$, so from now on we assume that $m>\frac{d}{d+2}$. We shall prove the result under the assumption $d\rightarrow\infty$ with $\alpha=\frac{1}{2-d(1-m)}$ being constant. We notice that, in such a limit, $m\rightarrow1$, however, $m$ may follow completely different paths.

Let us begin with the Wasserstein distance.
From \eqref{eq:FDEFPv} and \eqref{eq:vinf}, the solution $v(t,\cdot)$, seen as a probability density, is the image of the scaled Barenblatt profile $v_\infty(\cdot)$, seen as a probability density, by the affine map $T(x)=ax+h$ with 
\begin{equation}
    a=a(t):=(1-\mathrm{e}^{-\frac{t}{\alpha}})^\alpha
    \quad\text{and}\quad 
    h=h(t):=\mathrm{e}^{-t}x_0.
\end{equation}
Therefore by Lemma \ref{le:wasposca}, the computation of the distance boils down to the first two moments: 
\begin{equation}
    \mathrm{W}_2^2(v(t,\cdot),v_\infty(\cdot))
    =(a-1)^2m_2+2(a-1)\langle m_1,h\rangle+|h|^2
\end{equation}
where
\begin{equation}
    m_1:=\int xv_\infty(x)\mathrm{d}x  =0\quad\mbox{and}\quad m_2:=\int|x|^2v_\infty(x)\mathrm{d}x
\end{equation}
Fix $\alpha \in\left(\frac{1}{2}, \infty\right)$, by setting $m=\frac{\alpha(d-2)+1}{\alpha d}$, $p=\frac{1}{1-m}=\frac{d\alpha}{2\alpha - 1}$ and $b=\frac{2\alpha-1}{2(\alpha(d-2)+1)}$ we have that $m_2=M_2$ where $M_2$ is as in~\eqref{eq:Momentsfast}. 
Hence,
\begin{equation} \label{eq:w2_exact_m>1}
    \mathrm{W}_2^2(v(t,\cdot),v_\infty(\cdot))
    =(a-1)^2M_2+|h|^2
    =M_2((1-\mathrm{e}^{-\frac{t}{\alpha}})^\alpha-1)^2+|x_0|^2\mathrm{e}^{-2t}.
\end{equation}
By taking into account formulas~\eqref{eq:M2sim} and by expanding $a(t)$, one finds that
\begin{equation}\label{final-inequality-w2}
    \begin{split}
\mathrm{W}_2^2(v(t,\cdot),v_\infty(\cdot)) &= \left((2\pi\mathrm{e})^{2\alpha-1}d+o_{d\rightarrow\infty}(d)\right)\left(-\alpha \mathrm{e}^{-\frac{t}{\alpha}}+o_{t\rightarrow\infty}(\mathrm{e}^{-\frac{1}{\alpha}t})\right)^2 + |x_0|^2\mathrm{e}^{-2t} \\
&=\alpha^2\,(2\pi\mathrm{e})^{2\alpha-1}d  \mathrm{e}^{-\frac{2}{\alpha}t} + |x_0|^2\mathrm{e}^{-2t} + do_{t\rightarrow\infty}(\mathrm{e}^{-\frac{2}{\alpha}t}) + o_{d\rightarrow\infty}(d)o_{t\rightarrow\infty}(\mathrm{e}^{-\frac{2}{\alpha}t}) \\ 
& + \mathrm{e}^{-\frac{2}{\alpha}t}o_{d\to\infty}(d)\,.
    \end{split}
\end{equation}
Let $\varepsilon\in(0,1)$, and set $t_d=\frac{\max(1,\alpha)}{2}(1-\varepsilon)\log(d)$. 
Let us first consider the case $\alpha \geq 1$. By a straightforward computation find that  $\alpha^2\,(2\pi\mathrm{e})^{2\alpha-1}d  \mathrm{e}^{-\frac{2}{\alpha}t_d}$ diverges to $+\infty$. Since the other terms in~\eqref{final-inequality-w2} (excluding $|x_0|^2 \mathrm{e}^{-2t_d}$ which may give a positive contribution) are of a smaller order, this is enough to prove the upper bound on the cutoff for $W_2^2$. 
On the other hand, if $\alpha < 1$, it is  $\sup_{|x_0| \leq r\sqrt{d}} |x_0|^2 \mathrm{e}^{-2t_d}$ which diverges to $+\infty$, and provides the upper bound. 

Consider now instead $t_d=\frac{\max(1,\alpha)}{2}(1+\varepsilon)\log(d)$, in such a case both the term $\alpha^2\,(2\pi\mathrm{e})^{2\alpha-1}d  \mathrm{e}^{-\frac{2}{\alpha}t_d}$ and $|x_0|^2 \mathrm{e}^{-2t_d}$ converge to zero as $d\rightarrow\infty$. This proves the cutoff for $W_2^2$, since the other terms in~\eqref{final-inequality-w2} are of a smaller order.

Let us now consider $\mathrm{H}_m$ defined in~\eqref{eq:tsalis}. Assume $\alpha$, $m$, $p$ and $b$ as before, the following simple computations
\begin{equation}\label{needed.computations}
\int_{\mathbb{R}^d} v^m(t)\mathrm{d}x =  a^\frac{2\alpha-1}{\alpha} N_m \quad\mbox{and}\quad
\int_{\mathbb{R}^d} |x|^2v(t)\mathrm{d}x= |h|^2 + a^2\, M_2\,,
\end{equation}
show that
\begin{equation*}
 \mathrm{H}_m(v(t)\mid v_\infty) =\frac{\alpha d}{2\alpha-1}\left(1-a^\frac{2\alpha-1}{\alpha}\right) 
 N_m +\frac{a^2-1}{2}M_2 + \frac{\mathrm{e}^{-2t}|x_0|^2}{2}\,,
\end{equation*}
where $M_2$ is as before and $N_m$ defined in~\eqref{eq:Bmnorm}. By expanding $a(t)$ at the second order, and taking into account the behaviour~\eqref{difference-fde},  we find that
\begin{equation}\label{inequality-hm}
\begin{split}
     \mathrm{H}_m(v(t)\mid v_\infty)& =\frac{\alpha}{2} (2\pi\mathrm{e})^{2\alpha-1}d \mathrm{e}^{-\frac{2}{\alpha}t} + \frac{|x_0|^2 \mathrm{e}^{-2t}}{2}+  do_{t\rightarrow\infty}(\mathrm{e}^{-\frac{2}{\alpha}t}) \\ 
     & \quad + o_{d\rightarrow\infty}(d)o_{t\rightarrow\infty}(\mathrm{e}^{-\frac{2}{\alpha}t})  + O_{d\to\infty}(1) \mathrm{e}^{-\frac{t}{\alpha}}\,.
\end{split}
\end{equation}
We notice that the terms of largest order in~\eqref{inequality-hm} are the same as in~\eqref{final-inequality-w2}, therefore we can conclude in the same way as for the $W_2^2$ distance.

Let us consider the Fisher Information defined in~\eqref{eq:fisher}. By identity~\eqref{eq:vinf} we have that
\begin{equation*}
    v^{m-1}(t,x)=a^{d(1-m)}\left(C+\frac{1-m}{2m}\left|\frac{x-h}{a}\right|^2\right)\quad\mbox{and}\quad \frac{m}{m-1}\nabla v^{m-1} =-a^{-\frac{1}{\alpha}}\left(x-h\right)
\end{equation*}
Therefore, by taking into account~\eqref{needed.computations} and that $\int_{\mathbb{R}^d} x v(t,x) d\mathrm{d}x = h$, we have
\begin{equation*}
    \begin{split}
         \mathrm{I}_m(v(t)\mid v_\infty) & =\int_{\mathbb{R}^d}v(t,x) \left|x-a^{-\frac{1}{\alpha}}(x-h)\right|^2 \mathrm{d}x 
          \\ & = (1-a^{-\frac{1}{\alpha}})^2\left(|h|^2 + a^2\, M_2\right) + |h|^2\,a^{-\frac{2}{\alpha}} +2 |h|^2\,a^{-\frac{1}{\alpha}} (1-a^{-\frac{1}{\alpha}})  \\
          & = |h|^2 + a^2 (1-a^{-\frac{1}{\alpha}})^2 M_2
    \end{split}
\end{equation*}
 By expanding $a(t)$ at the second order and taking into account~\eqref{eq:M2sim} we get
 \begin{equation*}
      \mathrm{I}_m(v(t)\mid v_\infty) = (2\pi\mathrm{e})^{2\alpha-1}d \mathrm{e}^{-\frac{2}{\alpha}t} + |x_0|^2 \mathrm{e}^{-2t} +  do_{t\rightarrow\infty}(\mathrm{e}^{-\frac{2}{\alpha}t}) + o_{d\rightarrow\infty}(d)o_{t\rightarrow\infty}(\mathrm{e}^{-\frac{2}{\alpha}t})\,.
 \end{equation*}
Since the expansion of $\mathrm{I}_m(v(t)\mid v_\infty)$ is very similar to the two distances used before, we can conclude in the same way. 
\end{proof}

\begin{remark}[Full-space Barenblatt profile and multivariate Student t-distribution]\label{rk:student}
In Statistics, the multivariate Student t-distribution is the law of the random vector of $\mathbb{R}^d$
\begin{equation}
X:=x_0+\frac{Y}{\sqrt{\frac{Z}{r}} }
\quad\text{where}\quad
\begin{cases}
 Y&\sim\mathcal{N}(0,\Sigma)\\
 Z&\sim\chi^2(r)=\mathrm{Gamma}(\frac{r}{2},\frac{1}{2})
 \end{cases}
 \quad\text{are independent}.
\end{equation} 
Here $x_0\in\mathbb{R}^d$ is a vector called the position, $r>0$ is a positive
real parameter called the degree of freedom, and $\Sigma$ is an $d\times d$
positive-definite symmetric matrix. The probability density function is
\begin{equation}
    x\in\mathbb{R}^d\mapsto
    \frac{C}{\bigr(1+\frac{1}{r}(x-x_0)^\top\Sigma^{-1}(x-x_0)\bigr)^{\frac{r+d}{2}}},
    \quad
    C=\frac{\Gamma(\frac{r+d}{2})}
    {\Gamma(\frac{r}{2})\sqrt{\det(r\pi\Sigma)}}.
\end{equation}
It has a mean (respectively covariance) if and only if $r>1$ (respectively\ $r>2$), given respectively by 
\begin{equation}\label{eq:x0Sigma}
x_0\quad \text{and}\quad \frac{r}{r-2}\Sigma.
\end{equation}
The case $r=1$ is also known as a multivariate Cauchy distribution. At fixed $d$, the multivariate Student t-distribution tends as $r\to\infty$ to $\mathcal{N}(x_0,\Sigma)$, thanks to the law of large numbers and the Slutsky lemma. In the isotropic case $\Sigma=\sigma^2\mathrm{id}_d$, the law of $X$ is rotationally invariant around $x_0$. The real random variable $\frac{1}{d}(X-x_0)^\top\Sigma^{-1}(X-x_0)$ follows the Fisher\,--\,Snedecor F-distribution $F(d,r)$, since $(X-x_0)^\top\Sigma^{-1}(X-x_0)=r\frac{|\Sigma^{-1/2}Y|^2}{Z}$ with $|\Sigma^{-1/2}Y|^2\sim\chi^2(d)$ independent of $Z\sim\chi^2(r)$. In the isotropic case $\Sigma=\sigma^2\mathrm{id}_d$ with $\sigma^2>0$, the density becomes
\begin{equation}
      \frac{C}{\bigr(1+\frac{1}{r\sigma^2}\left|x-x_0\right|^2\bigr)^{\frac{r+d}{2}}}
    = \frac{1}{(C^{-\frac{2}{r+d}}+C^{-\frac{2}{r+d}}\frac{1}{r\sigma^2}\left|x-x_0\right|^2)^{\frac{r+d}{2}}},
    \quad
    C=\frac{\Gamma(\frac{r+d}{2})}
    {\Gamma(\frac{r}{2})\sqrt{r^d\pi^d\sigma^{2d}}},
\end{equation}
which is a Barenblatt profile. In view of \eqref{eq:sol}, we have $\frac{r+d}{2}=\frac{1}{1-m}$, hence $r=\frac{1}{\alpha(1-m)}$ where $\alpha:=\frac{1}{2-d(1-m)}$, and we choose $\sigma^2$ in such a way that $C^{-\frac{2}{r+d}}\frac{1}{r\sigma^2}=\alpha\frac{1-m}{2m}$. In particular, when $m=\frac{d-1}{d}$, then $r=d$, and since $M_2=\mathrm{Tr}(\frac{d}{d-2}\sigma^2I_d)=\frac{d^2}{d-2}\sigma^2$, we get $\sigma^2\sim_{d\to\infty}2\pi\mathrm{e}$ from \eqref{eq:M2sim}.
\end{remark}

\section{Compactly supported Barenblatt and proof of Theorems \ref{th:main} (porous medium case) and Theorem \ref{th:fixed_exp}}

This section is devoted to the proof of Theorem \ref{th:main} in the porous medium case $0<\alpha<1/2$, which shall be followed by the proof of Theorem \ref{th:fixed_exp}.

\begin{lemma}[Compactly supported Barenblatt pdf]\label{le:Bm>1}
 The compactly supported Barenblatt probability density function in dimension $d\geq1$ and with shape parameter $p>0$ and scale parameter $b>0$ is 
 \begin{equation}\label{eq:CBxc}
   B(x):=(c-b|x|^2)_+^p,\quad x\in\mathbb{R}^d,\quad\text{where}\quad
   c:=\left(\frac{b^{d/2}\Gamma\left(p + 1 + \frac{d}{2}\right)}{\pi^{d/2}\Gamma(p+1)}\right)^{2/(2p+d)}.
 \end{equation}
 We also have 
 \begin{equation} \label{eq:CBac}
   \int_{\mathbb{R}^d}|x|^aB(x)\mathrm{d}x
   =\left(\frac{\Gamma\left(p+1+\frac{d}{2}\right)}{\pi^{d/2}b^p\Gamma(p+1)}\right)^{a/(d+2p)}\frac{\Gamma\left(p+1+\frac{d}{2}\right)\Gamma\left(\frac{d+a}{2}\right)}{\Gamma\left(\frac{d}{2}\right)\Gamma\left(p+1+\frac{d+a}{2}\right)},
   \quad\text{for all $a\geq0$}.
 \end{equation}
 and
 \begin{equation} \label{eq:Bmpme}
     \int_{\mathbb{R}^d}B^m(x)\mathrm{d}x =  \frac{\Gamma(p+1+\frac{d}{2})}{\Gamma(pm+1+\frac{d}{2})} \frac{\Gamma(pm+1)}{\Gamma(p+1)} \left(\frac{b^\frac{d}{2}}{\pi^\frac{d}{2}} \frac{\Gamma(p+1+\frac{d}{2})}{\Gamma(p+1)}\right)^\frac{2p(m-1)}{2p+d}\quad\text{for all $m>0$}.
 \end{equation}
If we fix $\alpha\in\left(0, \frac{1}{2}\right)$, and set $m=1+\frac{1-2\alpha}{d\alpha}$, $p=\frac{1}{m-1}$, $b = \frac{m-1}{2m}$, then
\begin{equation}\label{limit-alpha-constant-pme}
    \int_{\mathbb{R}^d}|x|^2B(x)\mathrm{d}x = \frac{d}{\left(2\pi\mathrm{e}\right)^{1-2\alpha}} + o_{d\to\infty}(d)\quad\mbox{and}\quad \int_{\mathbb{R}^d}B^m(x)\mathrm{d}x = \frac{1}{\left(2\pi\mathrm{e}\right)^{1-2\alpha}}+o_{d\to\infty}(1)\,.
\end{equation}
Furthermore, under the same assumptions
\begin{equation}\label{difference-pme}
    \int_{\mathbb{R}^d}|x|^2 B(x)\mathrm{d}x - d \int_{\mathbb{R}^d}B(x)^m\mathrm{d}x = O_{d\to\infty}(1)
\end{equation}
\end{lemma}

\begin{proof}
As in the proof of Lemma~\ref{le:Bm<1}, by an Euler-Beta integral, we get, for all $d\ge1$, $p>0$, $a\ge0$, $b>0$ and $c>0$,
\begin{align}\label{eq:Baporous}
M_a:=\int{|x|^aB(x)\mathrm{d}x} &= c^{p + (d+a)/2}b^{-(d+a)/2}\frac{\pi^{d/2}}{\Gamma(d/2)}\int_0^1{(1-u)^p u^{(d+a)/2 - 1}du} \nonumber \\
&= \left(\frac{c}{b}\right)^{(a+d)/2}\frac{\Gamma(p+1)\Gamma((d+a)/2)}{\Gamma(d/2)\Gamma(p+1 + (d+a)/2)}c^p\pi^{d/2},
\end{align}
where we have used that the surface of the unit sphere of $\mathbb{R}^d$, $d\ge1$ is $2\pi^\frac{d}{2}/\Gamma(\frac{d}{2})$ and the fact that 
\[
\int_0^1 (1-t)^kt^l\mathrm{d}t=\frac{\Gamma(k+1)\Gamma(l+1)}{\Gamma(k+l+2)}.
\]
We notice that $B$ is a probability density if and only if $M_0=1$, which determines the value of $c$ given in~\eqref{eq:CBxc}. The value of $M_a$ in~\eqref{eq:CBac} is then computed accordingly to the determined value of $c$. Finally, we get~\eqref{eq:Bmpme}, by using $M_0$ with $pm$ instead of $p$ and with the value of $c$ given in~\eqref{eq:CBxc}.

Let us now consider the limit for $d\rightarrow\infty$ with $\alpha$ being constant, that is, we obtain the formulas~\eqref{limit-alpha-constant-pme}. Under such assumptions, we have that $m\to1$ as $d\to\infty$. We notice that the previous assumptions give us $\alpha=\frac{1}{2+d(m-1)}$. Therefore, by setting $\beta=1-2\alpha>0$, we have
\begin{equation}\label{m2-pme}    M_2 = \pi^{-d(m-1)\alpha} b^{-2\alpha} \left(\frac{\Gamma(1+\frac{d}{2\beta})}{\Gamma(1+\frac{d\alpha}{\beta})}\right)^{\frac{2\beta}{d}}\, \frac{\beta d}{d+2\beta}\,.
\end{equation}
By noticing that $d\alpha(m-1)=\beta$, $b^{-2\alpha}=\beta^{-2\alpha}(2\alpha d)^{2\alpha}+o_{d\rightarrow\infty}(d^{2\alpha})$ and
by using the following relation $\Gamma(1+sz)^\frac{1}{z}\sim_{z\rightarrow\infty} (\frac{zs}{e})^s$ which holds for any $s>0$, we find the left-hand side of~\eqref{limit-alpha-constant-pme}. Consider now
\begin{equation}\label{nm-pme}
N_m:=\int_{\mathbb{R}^d}B^m(x)\mathrm{d}x = \left(\frac{b}{\pi}\right)^{d(m-1)\alpha} \,\left(\frac{\Gamma(1+\frac{d}{2\beta})}{\Gamma(1+\frac{d\alpha}{\beta})}\right)^{\frac{2\beta}{d}}\,\frac{2d\alpha+2\beta}{d+2\beta} , 
\end{equation}
by a similar computation as before, one easily finds the right-hand side of~\eqref{limit-alpha-constant-pme}.

Let us now prove relation~\eqref{difference-pme}. Let us recall that $b=\frac{\beta}{2d\alpha}+O_{d\to\infty}(d^{-2})$. We have, therefore: 
\begin{equation}
\begin{split}
M_2-dN_m& =\frac{M_2}{\beta}\left(\beta-b(2\alpha d+2\beta)\right) \\
& = \frac{M_2}{\beta} \left(\beta - \beta + O_{d\to\infty}(d^{-1})\right)= \frac{M_2}{\beta}O_{d\to\infty}(d^{-1})
\end{split}
\end{equation}
and we find relation~\eqref{difference-pme} since $M_2=O_{d\to\infty}(d)$ as $d\to\infty$.
\end{proof}

We also need the asymptotic for fixed $m$ to prove Theorem \ref{th:fixed_exp}. 

\begin{lemma} \label{lem:fixed.m}
    With the same notation as in Lemma \ref{le:Bm>1}, but fixing $m > 1$ and setting $\alpha = \frac{1}{2-d(1-m)}$, as well as $p=\frac{1}{m-1}$ and $b =\frac{m-1}{2m}$, we have as $d$ goes to infinity

    \begin{equation}
    c \sim \frac{bd}{2\pi\mathrm{e}},\quad
    \int_{\mathbb{R}^d}|x|^2B(x)\mathrm{d}x \sim \frac{d}{2\pi\mathrm{e}},\quad\text{and}\quad
    \int_{\mathbb{R}^d}B^m(x)\mathrm{d}x \sim \frac{1}{2\pi\mathrm{e}}.
    \end{equation}
\end{lemma}
In particular, for fixed $m$, as $d$ goes to infinity the invariant Barenblatt profile has support on a ball of radius of order $\sqrt{d}$. The proof is obtained by using the exact formulas \eqref{eq:CBac} and \eqref{eq:Bmpme}, Stirling's formula and the formula $\Gamma(x+t) \equiv \Gamma(x)x^t$ for fixed $t$, as $x$ goes to $+\infty$. They are compatible with Lemma \ref{le:Bm>1} when $\alpha \rightarrow 0$.

\begin{proof}[Proof of Theorem \ref{th:main} when $0<\alpha<1/2$ (porous medium case)]
Let us consider the case of $\alpha$ constant and $d\rightarrow\infty$. Consider first the Wasserstein distance, as in the proof of the fast diffusion case, the solution $v(t,\cdot)$, seen as a probability density, is the image of the scaled Barenblatt profile $v_\infty(\cdot)$, seen as a probability density, by the affine map $T(x)=ax+h$ with 
\begin{equation}
    a=a(t):=(1-\mathrm{e}^{-\frac{t}{\alpha}})^\alpha
    \quad\text{and}\quad 
    h=h(t):=\mathrm{e}^{-t}x_0.
\end{equation}
Therefore by Lemma \ref{le:wasposca}, the computation of the distance boils down to the first two moments, and, as in the proof for the fast diffusion case, we find 
\begin{equation} \label{eq:W2_exact_mgeq1}
    \mathrm{W}_2^2(v(t,\cdot),v_\infty(\cdot))
    =(a-1)^2M_2+|h|^2\,,
\end{equation}
where $M_2$ is as in~\eqref{m2-pme}. By developing $a(t)$ at the first order and by taking into account~\eqref{limit-alpha-constant-pme} we find
\[
\begin{split}
\mathrm{W}_2^2(v(t,\cdot),v_\infty(\cdot))
    & = \frac{\alpha^2}{(2\pi\mathrm{e})^{1-2\alpha}}\, d\, \mathrm{e}^{-\frac{2}{\alpha}t} + |x_0|^2 \mathrm{e}^{-2t}  \\
    & + d o_{t\to\infty}(\mathrm{e}^{-\frac{2}{\alpha}t})  + \mathrm{e}^{-\frac{2}{\alpha}t}o_{d\to\infty}(d) + o_{t\to\infty}(\mathrm{e}^{-\frac{2}{\alpha}t})o_{d\to\infty}(d)\,.
\end{split}
\]
The proof of formulas~\eqref{eq:m>1:W2} for the Wasserstein distance $W_2^2$ is as in the proof of the fast diffusion case, with $|x_0|^2 \mathrm{e}^{-2t}$ being the dominant term. Let us consider the the relative entropy  $\mathrm{H}_m$ defined in~\eqref{eq:tsalis}.
As in the proof of the fast diffusion case, we find that
\[
\mathrm{H}_m(v(t)\mid v_\infty)= \frac{\alpha\,d}{1-2\alpha} \left(a^\frac{2\alpha-1}{\alpha}-1\right)N_m +\frac{a^2-1}{2}M_2 + \frac{|h|^2}{2}\,.
\]
By considering formulas~\eqref{limit-alpha-constant-pme} and~\eqref{difference-pme},  by expanding $a(t)$ at the second order, we obtain
\begin{equation*}
\begin{split}
     \mathrm{H}_m(v(t)\mid v_\infty)& =\frac{\alpha}{2}\,\frac{d}{(2\pi\mathrm{e})^{1-2\alpha}} \mathrm{e}^{-\frac{2}{\alpha}t} + \frac{|x_0|^2 \mathrm{e}^{-2t}}{2}+  do_{t\rightarrow\infty}(\mathrm{e}^{-\frac{2}{\alpha}t}) \\ 
     & \quad + o_{d\rightarrow\infty}(d)o_{t\rightarrow\infty}(\mathrm{e}^{-\frac{2}{\alpha}t})  + O_{d\to\infty}(1) \mathrm{e}^{-\frac{t}{\alpha}}\,,
\end{split}
\end{equation*}
from which we conclude as in the fast diffusion case. 

Let us consider the Fisher Information defined in~\eqref{eq:fisher}, the only difference here with respect to the fast diffusion case is the compact support of $v$. By identities~\eqref{eq:FDEFPv} and~\eqref{eq:vinf}, we obtain
\begin{equation*}
    v^{m-1}(t,x)=a^{d(1-m)}\left(C-\frac{m-1}{2m}\left|\frac{x-h}{a}\right|^2\right)_{+}\quad\mbox{and}\quad \frac{m}{m-1}\nabla v^{m-1} =-a^{-\frac{1}{\alpha}}\left(x-h\right) \mathbf{1}_{\{v>0\}}\,,
\end{equation*}
where $\mathbf{1}_A$ is the characteristic function of the set $A\subset \mathbb{R}^d$. Proceeding as for the fast diffusion case, we find
\begin{equation*}
    \begin{split}
         \mathrm{I}_m(v(t)\mid v_\infty) & =\int_{\mathbb{R}^d}v(t,x) \left|\left(x-a^{-\frac{1}{\alpha}}(x-h)\right) \mathbf{1}_{\{v>0\}} + x\mathbf{1}_{\{v=0\}} \right|^2 \mathrm{d}x 
          \\ & = (1-a^{-\frac{1}{\alpha}})^2\left(|h|^2 + a^2\, M_2\right) + |h|^2\,a^{-\frac{2}{\alpha}} +2 |h|^2\,a^{-\frac{1}{\alpha}} (1-a^{-\frac{1}{\alpha}})  \\
          & = |h|^2 + a^2 (1-a^{-\frac{1}{\alpha}})^2 M_2\,,
    \end{split}
\end{equation*}
where we have used the following relations
\begin{equation*}
\begin{split}
\int_{\mathbb{R}^d} x\,v(t,x) \mathrm{d}x = h\,\quad&\mbox{and} \\ &\int_{\mathbb{R}^d} |x|^2 v(t,x) \mathbf{1}_{\{v=0\}} \mathrm{d}x = \int_{\mathbb{R}^d} v(t,x) \, x\cdot \left(x-a^{-\frac{1}{\alpha}}(x-h)\right) \mathbf{1}_{\{v>0\}}\mathbf{1}_{\{v=0\}} \mathrm{d}x =   0
\end{split}
\end{equation*}
By expanding $a(t)$ at the first order and taking into account~\eqref{limit-alpha-constant-pme} we get
 \begin{equation*}
      \mathrm{I}_m(v(t)\mid v_\infty) = \frac{d\,\mathrm{e}^{-\frac{2}{\alpha}t}}{(2\pi\mathrm{e})^{1-2\alpha}} + |x_0|^2 \mathrm{e}^{-2t} +  do_{t\rightarrow\infty}(\mathrm{e}^{-\frac{2}{\alpha}t}) + o_{d\rightarrow\infty}(d)o_{t\rightarrow\infty}(\mathrm{e}^{-\frac{2}{\alpha}t})\,,
 \end{equation*}
and we conclude as previously.
\end{proof}

\begin{proof}[Proof of Theorem \ref{th:fixed_exp}]
    Let us start with Wasserstein distance $W_2$. As in \eqref{eq:W2_exact_mgeq1}, we have
    \begin{equation} \label{eq:aux:fixed:m}
     \mathrm{W}_2^2(v(t,\cdot),v_\infty(\cdot))
    =M_2((1-\mathrm{e}^{-\frac{t}{\alpha}})^\alpha-1)^2+|x_0|^2\mathrm{e}^{-2t}.
    \end{equation}
    Since the first term is nonnegative, the lower bound for cutoff immediately follows. For the upper bound, recalling that 
    \[
    \alpha = \frac{1}{2-d(1-m)} \sim \frac{1}{(m-1)d},
    \]
    we have for $t_d$ going to infinity with $d$ (noting that $t_d/\alpha$ also does)
    \[
    (1-\mathrm{e}^{-\frac{t_d}{\alpha}})^\alpha \sim \exp\left(-\alpha\mathrm{e}^{-t_d/\alpha}\right)
    \quad\text{so that}\quad
    \left((1-\mathrm{e}^{-\frac{t}{\alpha}})^\alpha-1\right)^2 \sim \frac{\mathrm{e}^{-2(m-1)d t_d}}{(m-1)^2d^2}.
    \]
    Since $M_2$ is of order $d$ by Lemma \ref{lem:fixed.m}, it is easy to check that when $t_d$ is of order $\log d$ and $|x_0|$ of order $\sqrt{d}$ it is the second term that dominates in \eqref{eq:aux:fixed:m}, and the upper bound in the cutoff phenomenon easily follows. 

    The proofs for entropy and Fisher information work in the same way : we start from the exact formulas, and there is always a term proportional to $|x_0|^2\mathrm{e}^{-2t}$ which dominates the other ones. 
\end{proof}

\begin{remark}[Compactly supported Barenblatt profile and projected spherical
  law]\label{rk:spherical}
  If $X$ is a random vector of $\mathbb{R}^d$, $d\geq2$, uniformly distributed
  on the centered sphere $\{x\in\mathbb{R}^d:|x|=R\}$ of radius $R>0$, then
  for all $1\leq n\leq d-1$, the law of the random vector $Y:=(X_1,\ldots,X_n)$
  of $\mathbb{R}^n$ has density
  \[
    y\in\mathbb{R}^n\mapsto
    C(R^2-|y|^2)_+^{\frac{d-n-2}{2}}
    =C(R-|y|)^{\frac{d-n-2}{2}}(R+|y|)^{\frac{d-n-2}{2}}\mathbf{1}_{\{|y|\leq R\}}
  \]
  where $C:=\frac{\Gamma(\frac{d}{2})}{\pi^{\frac{n}{2}}\Gamma(\frac{d-n}{2})R^{d-2}}=\frac{\Gamma(\frac{n}{2})}{\pi^{\frac{n}{2}}R^{d-2}\mathrm{Beta}(\frac{n}{2},\frac{d-n}{2})}$. We
  recognize, for $n\leq d-3$, a compactly supported Barenblatt profile with shape parameter
  $p=\frac{d-n-2}{2}>0$, in other words a special radial or multivariate
  symmetric Beta distribution. For $n=d-2$, the law of $Y$ is uniform on the
  ball of radius $R$, while for $n=d-1$ the exponent is $-\frac{1}{2}$ and the
  density blows up at the boundary, so it is not a Barenblatt profile. When $n=1$, this is also known as the
  Funk\,--\,Hecke formula in Harmonic Analysis. We have $Y=\pi_n(X)$ where
  $\pi_n(x_1,\ldots,x_d)=(x_1,\ldots,x_n)$ is the projection of $\mathbb{R}^d$
  on the first $n$ coordinates. Note that if we use instead the stereographic
  projection $x\mapsto\frac{(x_1,\ldots,x_{d-1})}{1-x_d}$ of the unit sphere,
  then we end up with the radially symmetric distribution on
  $\mathbb{R}^{d-1}$ of density
  $\frac{\Gamma(d-1)}{\pi^{\frac{d-1}{2}}\Gamma(\frac{d-1}{2})}(1+|y|^2)^{-(d-1)}$,
  which is exactly a full space Barenblatt profile, with shape parameter
  $p=d-1$. To make a link with the Gaussian construction in Remark \ref{rk:student}, if $Z$ follows the standard Gaussian distribution $\mathcal{N}(0,I_d)$ on $\mathbb{R}^d$ then $Y:=Z/|Z|$ follows the uniform distribution on the unit sphere of $\mathbb{R}^d$, and $R(Y_1,\ldots,Y_n)$ follows the compactly supported Barenblatt profile above. As in Remark \ref{rk:student}, here we divide a Gaussian vector $Z$ by $|Z|\sim\chi(d)$ but this time these two objects are not independent.
  Compactly supported Barenblatt profiles also appear in potential theory in the works of Marcel Riesz for example, see for instance \cite{zbMATH07559881,zbMATH07559881-post} and references therein in relation with Mellin transform.
\end{remark}


\begin{thebibliography}{10}

\bibitem{Aldous83}
D.~Aldous.
\newblock Random walks on finite groups and rapidly mixing {Markov} chains.
\newblock Semin. de probabilites {XVII}, {Proc}. 1981/82, {Lect}. {Notes}
  {Math}. 986, 243-297 (1983)., 1983.

\bibitem{zbMATH04096011}
S.~Angenent.
\newblock Large time asymptotics for the porous media equation.
\newblock Nonlinear diffusion equations and their equilibrium states {I},
  {Proc}. {Microprogram}, {Berkeley}/{Calif}. 1986, {Publ}., {Math}. {Sci}.
  {Res}. {Inst}. 12, 21-34 (1988)., 1988.

\bibitem{MR0046217}
G.~I. Barenblatt.
\newblock On some unsteady motions of a liquid and gas in a porous medium.
\newblock {\em Akad. Nauk SSSR. Prikl. Mat. Meh.}, 16:67--78, 1952.

\bibitem{MR1426127}
G.~I. Barenblatt.
\newblock {\em Scaling, self-similarity, and intermediate asymptotics},
  volume~14 of {\em Cambridge Texts in Applied Mathematics}.
\newblock Cambridge University Press, Cambridge, 1996.
\newblock With a foreword by Ya.\ B. Zeldovich.

\bibitem{MR3770869}
G.~Barrera.
\newblock Abrupt convergence for a family of {O}rnstein-{U}hlenbeck processes.
\newblock {\em Braz. J. Probab. Stat.}, 32(1):188--199, 2018.

\bibitem{MR2260742}
J.~Barrera, B.~Lachaud, and B.~Ycart.
\newblock Cut-off for {$n$}-tuples of exponentially converging processes.
\newblock {\em Stochastic Process. Appl.}, 116(10):1433--1446, 2006.

\bibitem{MR1469575}
S.~Benachour, P.~Chassaing, B.~Roynette, and P.~Vallois.
\newblock Processus associ\'es \`a{} l'\'equation des milieux poreux.
\newblock {\em Ann. Scuola Norm. Sup. Pisa Cl. Sci. (4)}, 23(4):793--832, 1996.

\bibitem{zbMATH06133917}
M.~Bonforte, J.~Dolbeault, G.~Grillo, and J.~L. V{\'a}zquez.
\newblock Sharp rates of decay of solutions to the nonlinear fast diffusion
  equation via functional inequalities.
\newblock {\em Proc. Natl. Acad. Sci. USA}, 107(38):16459--16464, 2010.

\bibitem{bonforte:hal-02887010}
M.~Bonforte, J.~Dolbeault, B.~Nazaret, and N.~Simonov.
\newblock {\em {Stability in Gagliardo-Nirenberg-Sobolev inequalities: flows,
  regularity and the entropy method}}.
\newblock Memoirs of the American Mathematical Society. AMS, 2022.

\bibitem{MR4528974}
J.~Boursier, D.~Chafa\"{\i}, and C.~Labb\'{e}.
\newblock Universal cutoff for {D}yson {O}rnstein {U}hlenbeck process.
\newblock {\em Probab. Theory Related Fields}, 185(1-2):449--512, 2023.

\bibitem{MR1100809}
Y.~Brenier.
\newblock Polar factorization and monotone rearrangement of vector-valued
  functions.
\newblock {\em Comm. Pure Appl. Math.}, 44(4):375--417, 1991.

\bibitem{caputo-labbe-lacoin}
P.~Caputo, C.~Labbé, and H.~Lacoin.
\newblock Cutoff phenomenon in nonlinear recombinations.
\newblock preprint
  \href{https://arXiv.org/abs/2402.11396v1}{arXiv:2402.11396v1}.

\bibitem{MR1986060}
J.~A. Carrillo and J.~L. V\'{a}zquez.
\newblock Fine asymptotics for fast diffusion equations.
\newblock {\em Comm. Partial Differential Equations}, 28(5-6):1023--1056, 2003.

\bibitem{zbMATH07559881}
D.~Chafa{\"{\i}}, E.~B. Saff, and R.~S. Womersley.
\newblock On the solution of a {Riesz} equilibrium problem and integral
  identities for special functions.
\newblock {\em J. Math. Anal. Appl.}, 515(1):29, 2022.
\newblock Id/No 126367.

\bibitem{zbMATH07559881-post}
D.~Chafa{\"{\i}}, E.~B. Saff, and R.~S. Womersley.
\newblock On the solution of a {Riesz} equilibrium problem and integral
  identities for special functions.
\newblock post-publication version
  \href{https://arxiv.org/abs/2108.00534v6}{arXiv:2108.00534v6}, 2022.

\bibitem{chafathi}
D.~Chafaï and M.~Fathi.
\newblock On cutoff via rigidity for high dimensional curved diffusions.
\newblock preprint \href{http://arxiv.org/abs/2412.15969v1}{arXiv:2412.15969v1}
  to appear in {C}omptes {R}endus {M}athématique (2025), 2024.

\bibitem{MR2375599}
G.-Y. Chen and L.~Saloff-Coste.
\newblock The cutoff phenomenon for ergodic {M}arkov processes.
\newblock {\em Electron. J. Probab.}, 13:no. 3, 26--78, 2008.

\bibitem{MR2584746}
G.-Y. Chen and L.~Saloff-Coste.
\newblock The {$L^2$}-cutoff for reversible {M}arkov processes.
\newblock {\em J. Funct. Anal.}, 258(7):2246--2315, 2010.

\bibitem{MR1940370}
M.~Del~Pino and J.~Dolbeault.
\newblock Best constants for {G}agliardo-{N}irenberg inequalities and
  applications to nonlinear diffusions.
\newblock {\em J. Math. Pures Appl. (9)}, 81(9):847--875, 2002.

\bibitem{demange-thesis}
J.~Demange.
\newblock {\em Des équations à diffusion rapide aux inégalités de Sobolev
  sur les modèles de la géométrie}.
\newblock PhD thesis, Toulouse, 2005.

\bibitem{MR2178944}
J.~Demange.
\newblock Porous media equation and {S}obolev inequalities under negative
  curvature.
\newblock {\em Bull. Sci. Math.}, 129(10):804--830, 2005.

\bibitem{MR2381156}
J.~Demange.
\newblock Improved {G}agliardo-{N}irenberg-{S}obolev inequalities on manifolds
  with positive curvature.
\newblock {\em J. Funct. Anal.}, 254(3):593--611, 2008.

\bibitem{zbMATH06431289}
J.~Denzler, H.~Koch, and R.~J. McCann.
\newblock {\em Higher-order time asymptotics of fast diffusion in {Euclidean}
  space: a dynamical systems approach}, volume 1101 of {\em Mem. Am. Math.
  Soc.}
\newblock Providence, RI: American Mathematical Society (AMS), 2015.

\bibitem{zbMATH06622527}
J.~Denzler, H.~Koch, and R.~J. McCann.
\newblock Long-time asymptotic expansions for nonlinear diffusions in
  {Euclidean} space.
\newblock In {\em Mathematical congress of the Americas. First mathematical
  congress of the Americas, Guanajuato, M\'exico, August 5--9, 2013}, pages
  85--94. Providence, RI: American Mathematical Society (AMS), 2016.

\bibitem{zbMATH02187606}
J.~Denzler and R.~J. McCann.
\newblock Phase transitions and symmetry breaking in singular diffusion.
\newblock {\em Proc. Natl. Acad. Sci. USA}, 100(12):6922--6925, 2003.

\bibitem{zbMATH02155995}
J.~Denzler and R.~J. McCann.
\newblock Fast diffusion to self-similarity: {Complete} spectrum, long-time
  asymptotics, and numerology.
\newblock {\em Arch. Ration. Mech. Anal.}, 175(3):301--342, 2005.

\bibitem{Diaconis96}
P.~Diaconis.
\newblock The cutoff phenomenon in finite {Markov} chains.
\newblock {\em Proc. Natl. Acad. Sci. USA}, 93(4):1659--1664, 1996.

\bibitem{arXiv:2209.08651}
J.~Dolbeault, M.~J. Esteban, A.~Figalli, R.~L. Frank, and M.~Loss.
\newblock Sharp stability for {Sobolev} and log-{Sobolev} inequalities, with
  optimal dimensional dependence.
\newblock Preprint, {arXiv}:2209.08651 [math.{AP}] (2022), 2022.

\bibitem{MR3103175}
J.~Dolbeault and G.~Toscani.
\newblock Improved interpolation inequalities, relative entropy and fast
  diffusion equations.
\newblock {\em Ann. Inst. H. Poincar\'{e} C Anal. Non Lin\'{e}aire},
  30(5):917--934, 2013.

\bibitem{MR2421181}
A.~Figalli and R.~Philipowski.
\newblock Convergence to the viscous porous medium equation and propagation of
  chaos.
\newblock {\em ALEA Lat. Am. J. Probab. Math. Stat.}, 4:185--203, 2008.

\bibitem{MR0752258}
C.~R. Givens and R.~M. Shortt.
\newblock A class of {W}asserstein metrics for probability distributions.
\newblock {\em Michigan Math. J.}, 31(2):231--240, 1984.

\bibitem{MR1027932}
M.~Inoue.
\newblock Construction of diffusion processes associated with a porous medium
  equation.
\newblock {\em Hiroshima Math. J.}, 19(2):281--297, 1989.

\bibitem{MR2246356}
Y.~J. Kim and R.~J. McCann.
\newblock Potential theory and optimal convergence rates in fast nonlinear
  diffusion.
\newblock {\em J. Math. Pures Appl. (9)}, 86(1):42--67, 2006.

\bibitem{MR0745785}
M.~Knott and C.~S. Smith.
\newblock On the optimal mapping of distributions.
\newblock {\em J. Optim. Theory Appl.}, 43(1):39--49, 1984.

\bibitem{MR2203823}
B.~Lachaud.
\newblock Cut-off and hitting times of a sample of {O}rnstein-{U}hlenbeck
  processes and its average.
\newblock {\em J. Appl. Probab.}, 42(4):1069--1080, 2005.

\bibitem{Lacoin16}
H.~Lacoin.
\newblock Mixing time and cutoff for the adjacent transposition shuffle and the
  simple exclusion.
\newblock {\em Ann. Probab.}, 44(2):1426--1487, 2016.

\bibitem{LuPe16}
E.~Lubetzky and Y.~Peres.
\newblock Cutoff on all {Ramanujan} graphs.
\newblock {\em Geom. Funct. Anal.}, 26(4):1190--1216, 2016.

\bibitem{LuSl13}
E.~Lubetzky and A.~Sly.
\newblock Cutoff for the {Ising} model on the lattice.
\newblock {\em Invent. Math.}, 191(3):719--755, 2013.

\bibitem{MR1369395}
R.~J. McCann.
\newblock Existence and uniqueness of monotone measure-preserving maps.
\newblock {\em Duke Math. J.}, 80(2):309--323, 1995.

\bibitem{Meliot14}
P.-L. M{\'e}liot.
\newblock The cut-off phenomenon for {Brownian} motions on compact symmetric
  spaces.
\newblock {\em Potential Anal.}, 40(4):427--509, 2014.

\bibitem{MR2824568}
S.-i. Ohta and A.~Takatsu.
\newblock Displacement convexity of generalized relative entropies.
\newblock {\em Adv. Math.}, 228(3):1742--1787, 2011.

\bibitem{MR0114505}
R.~E. Pattle.
\newblock Diffusion from an instantaneous point source with a
  concentration-dependent coefficient.
\newblock {\em Quart. J. Mech. Appl. Math.}, 12:407--409, 1959.

\bibitem{pedrotti2025transportapproachcutoffphenomenon}
F.~Pedrotti and J.~Salez.
\newblock A transport approach to the cutoff phenomenon.
\newblock preprint \href{https://arXiv.org/abs/2509.08560v1}, 2025.

\bibitem{Salez25}
J.~Salez.
\newblock Cutoff for non-negatively curved diffusions.
\newblock preprint
  \href{https://arXiv.org/abs/2501.01304v1}{arXiv:2501.01304v1}, 2025.

\bibitem{MR1306030}
L.~Saloff-Coste.
\newblock Precise estimates on the rate at which certain diffusions tend to
  equilibrium.
\newblock {\em Math. Z.}, 217(4):641--677, 1994.

\bibitem{MR3409718}
F.~Santambrogio.
\newblock {\em Optimal transport for applied mathematicians}, volume~87 of {\em
  Progress in Nonlinear Differential Equations and their Applications}.
\newblock Birkh\"auser/Springer, Cham, 2015.
\newblock Calculus of variations, PDEs, and modeling.

\bibitem{zbMATH06262313}
C.~Seis.
\newblock Long-time asymptotics for the porous medium equation: the spectrum of
  the linearized operator.
\newblock {\em J. Differ. Equations}, 256(3):1191--1223, 2014.

\bibitem{zbMATH05058450}
J.~L. V{\'a}zquez.
\newblock {\em Smoothing and decay estimates for nonlinear diffusion equations.
  {Equations} of porous medium type}, volume~33 of {\em Oxf. Lect. Ser. Math.
  Appl.}
\newblock Oxford: Oxford University Press, 2006.

\bibitem{MR2286292}
J.~L. V\'azquez.
\newblock {\em The porous medium equation}.
\newblock Oxford Mathematical Monographs. The Clarendon Press, Oxford
  University Press, Oxford, 2007.
\newblock Mathematical theory.

\bibitem{MR1964483}
C.~Villani.
\newblock {\em Topics in optimal transportation}, volume~58 of {\em Graduate
  Studies in Mathematics}.
\newblock American Mathematical Society, Providence, RI, 2003.

\end{thebibliography}

\end{document}